\documentclass{article}
\usepackage{biblatex}
\usepackage{rotating}

\usepackage{amsmath, amssymb,amscd,geometry,amsthm}
\usepackage{thmtools}
\usepackage{color}
\usepackage[usenames,dvipsnames,svgnames,table]{xcolor}
\usepackage{epsfig}
\usepackage{soul}
\setstcolor{red}
\sethlcolor{yellow}

\usepackage{enumitem}
\usepackage{cprotect}
\usepackage{tikz}
\usepackage{quiver}
\usepackage{hyperref}

\theoremstyle{definition}

\newtheorem{thm}{Theorem}[section]
\newtheorem{propo}{Proposition}[section]
\newtheorem{defn}{Definition}[section]

\newtheorem{lemma}{Lemma}[section]

\renewcommand{\iff}{\leftrightarrow}

\newcommand\CC{\mathcal{C}}

\newcommand\RR{\mathbb{R}}

\newcommand\PP{\mathbb{P}}

\input{fitch.sty}
\numberwithin{equation}{subsection}
\usepackage{lipsum}
\bibliography{references}

\usepackage[utf8]{inputenc}

\input xy
\xyoption{all}

\title{On Rational Jurisprudence: A Problem in Bayesian Confirmation Theory}
\author{Reid Dale \\ \href{mailto:reiddale@berkeley.edu}{reiddale@berkeley.edu} }
\date{September 2022}

\begin{document}
\maketitle

\begin{abstract}
    This paper is concerned with the epistemic question of \textit{confirming} a hypothesis---the guilt of a defendant---by way of testimony heard by a juror over the course of an American-style criminal trial. In it, I attempt to settle a dispute between two strands of the legal community over the issue of whether the methods of Bayesian rationality are incompatible with jurisprudential principles such as the Presumption of Innocence. To this end, I prove a representation theorem that shows that so long as a juror would \textit{not} convict the defendant having heard no testimony (the Presumption of Innocence) but would convict upon hearing \textit{some} collection of testimony (Willingness to Convict), then this juror's disposition to convict the defendant is representable as the disposition of a Bayesian threshold juror in Posner's sense. This result indicates that relevant notion of a Bayesian threshold juror is insufficiently specified to render this debate a substantive one.
\end{abstract}

\section{Introduction} 

In \textit{State v. Skipper} \autocite{1994skipper}, the Supreme Court of Connecticut ruled that Bayesianism directly conflicts with the presumption of innocence, stating:
\begin{quote}
Because Bayes' Theorem requires the assumption of a prior probability of paternity, i.e., guilt, its use is inconsistent with the presumption of innocence in a criminal case such as this.... If we assume that the presumption of innocence standard would require the prior probability of guilt to be zero, the probability of paternity in criminal cases would always be zero.... In other words, Bayes' Theorem can only work if the presumption of innocence disappears from  considerations. \autocites(at 623){1994skipper} 
\end{quote}

Thus, the court argues, the jury cannot simultaneously hold the presumption of innocence and update their credences according to Bayes' rule without trivializing the enterprise of criminal trial by licensing \textit{only} verdicts of ``not guilty.'' Committed to the presumption of innocence, the court therefore rejects the use of Bayesian inference by the jury in a criminal setting. This caution is shared by many legal scholars, most notably Tribe \autocite{tribe1970trial}, who argues that many forms of probabilistic reasoning in the trial setting---including Bayesian inference---violate the presumption of innocence.

On the other hand, the economic analysis of law undertaken by Judge Richard Posner \autocite{posner1998evidence} suggests that a trier of fact \textit{ought} to be Bayesian. Posner models the ideal juror as an instance of what I term a \textit{Bayesian threshold juror}: an agent who updates her credence of guilt according to Bayes' rule and moves to convict just in case, at the conclusion of the trial, the credence she assigns to guilt is above some threshold value $\theta$ sufficiently close to 1. To ameliorate the worries of those skeptical of the Bayesian paradigm's compliance with jurisprudential norms such as the presumption of innocence, Posner proposes a simple solution: require of the juror that her credence of guilt at the outset of the trial is exactly $50\%$. 

In the following sections I will argue that such restrictions on an ideal Bayesian juror fail to meaningfully constrain a juror's disposition to render a conviction or acquittal. My argument relies on a formalization of the notion of ``Bayesian juror,'' wherein a juror's disposition to convict based on an observed sequence of testimonies is modeled as a function
\[f: \mathcal{T}_S \to \{C, A\} \]
where $\mathcal{T}_S$ is a set of collections of testimony that can be presented in a court of law, $C$ represents ``conviction,'' and $A$ represents ``acquittal.'' The disposition function $f$ is interpreted as $f(T) = C$ just in case, having heard all and only testimony $T = \{t_1,\dots, t_n\}$ over the course of the trial, the juror would vote to convict. That the juror's disposition be a function of the collection of testimony heard over the course of the trial is essential to modeling American criminal trials, as this is taken to be a constitutional right of the defendant, as described in \textit{Turner v. Louisiana}:

\begin{quote} The requirement that a jury's verdict ‘must be based upon the evidence developed at the trial’ goes to the fundamental integrity of all that is embraced in the constitutional concept of trial by jury...

In the constitutional sense, trial by jury in a criminal case necessarily implies at the very least that the ‘evidence developed’ against a defendant shall come from the witness stand in a public courtroom where there is full judicial protection of the defendant's right of confrontation, of cross-examination, and of counsel.\autocite{1965turner}
\end{quote} 

By testimony I refer not to the content of an agent's testimony but to information of the form ``Alice testified that \textit{S}, to which the defendant's attorney objected on grounds $X$, $Y$, and $Z$.'' Events such as this can effect a change in a juror's belief that \textit{S}, but the juror does not herself witness $S$. The fundamental assumption I make is
\begin{quote} 
\textbf{Testimonial Consistency Axiom (TCA):} Let $\mathcal{T}_S$ be a collection of possible testimonial events. For all subsets $T\in \mathcal{T}_S$, both the guilt and the innocence of the defendant are consistent with $T$.
\end{quote} 
For example, while the \textit{semantic content} of a witness' testimony may be inconsistent with either the guilt or innocence of the defendant, the witness' act of testifying to that effect is consistent with both guilt and innocence. After all, such a witness may be mistaken or lying. While there are some apparent violations of this assumption (e.g., a witness testifying that the defendant had murdered him), it generally holds true in the actual trial context. 

At minimum, the \textbf{Presumption of Innocence} (PoI) places the following constraint on a juror's disposition to convict at the outset of the trial: absent any testimony, the juror must not convict. In the notation of dispositions,
\[ f(\varnothing) = A.\]
Moreover, it is assumed that \textit{some} set of testimony would compel a juror to convict; in other words, there exists a set $T \in \mathcal{T}_S$ of testimony such that
\[ f(T) = C. \]
Call this the \textbf{Willingness to Convict} (WtC).

Using this formalism, I critically evaluate the interactions between Bayesianism and contemporary American legal theory. To this end, I focus on two key questions: 
\begin{enumerate} 
\item Are Posner's Bayesian threshold jurors rational agents in that they maximize the expectation of some utility function?
\item Does Posner's model of Bayesian threshold jurors materially constrain a juror's disposition beyond the (PoI) and (WtC)?
\end{enumerate}

Using recent work of Easwaran \autocite{easwarantruthlove}, I argue that the answer to Question 1 is ``Yes'' by exhibiting a utility function that a Bayesian threshold juror optimizes. However, by using the (TCA) I argue that the answer to Question 2 is ``No''; while Bayesian threshold jurors are rational qua the standard decision-theoretic account of rationality, \textit{all} dispositions satisfying (PoI) and (WtC) can be realized as the disposition of some Bayesian threshold Juror. 

This result calls into question the utility of modeling an ideal juror's inferential structure as a Bayesian threshold juror, as the disposition of \textit{any} juror satisfying (PoI) and (WtC) is rationalizable. For instance, consider the disposition that renders a conviction so long as at least two witnesses testify, regardless of the content of the testimony. This disposition satisfies the (PoI) and the (WtC), and by the representation theorem is represented by a Bayesian threshold juror. Such a person would be ill-suited to be a juror, yet the Posnerian account cannot rule them out.

A natural response to this result would be to add further constraints to rule out such contrived dispositions. In the final section of this paper multiple proposals in this vein are analyzed. As we will see, each of these proposals avail themselves to serious objections on both epistemic and jurisprudential grounds.

Ultimately, existing proposals to give probabilistic foundations to normative legal reasoning fail to do so, and the truth of principles such as (TCA) cast serious doubt that any such approach can succeed.

\section{A Formal Model of a Juror's Reasoning}

In this section I present the formal model, partially sketched in Section 1, of a juror's reasoning in a criminal trial. The atoms in this model are individual testimonies. I use the word ``testimony'' here very broadly: it can include any sensory information that a juror perceives during the course of the trial that may inform her rendering of a verdict, including the statements by witnesses, legal counsel, and judges made during the trial as well as qualitative information such as the demeanor and body language of the witness.

\subsection{Juror Dispositions}

A transcript $T$ of a trial is a collection of testimonies, to be understood as the contents of a trial as perceived by the juror. Given a collection $S$ of possible testimonies, the set $\mathcal{T}_S = \mathcal{P}(S)$ is the set of transcripts over $S$.

Recall from Section 1 that a juror's disposition to convict (hereafter ``disposition'') is simply a function
\[f: \mathcal{T}_S \to \{C,A\} \]
where $f(T) = C$ (respectively $A$) is read as ``the juror is disposed to convict (respectively acquit) on the basis of transcript $T$.''

So far, this model does not include any consideration of the material guilt or innocence of the defendant. To remedy this, we define a set of possible worlds relative to a set $S$ of possible testimonies by setting
\[W_S = \mathcal{T}_S \times \{G,I\},\]
where $G$ is shorthand for ``materially guilty'' and $I$ is shorthand for ``materially innocent'' of the charges alleged by the prosecution. In other words, a world is a pair $w = (T,x)$ consisting of a transcript $T$ and a value $x\in \{G,I\}$ corresponding to the material guilt or innocence of the defendant. It is critical to note that we are licensed in including both a world $(T,G)$ and $(T,I)$ by the axiom (TCA), which ensures that both material guilt and innocence are consistent with transcript $T$.

On $W_S$ we define two distinguished classes of events. First, for a transcript $T$ define 
\[ E_T = \{ (T,G), (T,I) \},\]
the two-world set consisting of a transcript $T$ and the two possible values for guilt, $G$ and $I$. Then $E_T$ is the event of receiving transcript $T$ from the trial. Second, define
\[E_G = \{w\,|\, w = (T,G) \text{ for some } T\in \mathcal{T}_S \}, \]
the collection of worlds where the defendant is materially guilty. Thus, $E_G$ is the event of the defendant being materially guilty.


\section{Bayesian Analysis in the Law} 

In this section some of the key literature surrounding the interaction between Bayesian epistemology and legal epistemology is reviewed from a formal perspective.

Throughout, I make two crucial assumptions. First, I assume that all evidence is presented as testimony. At first blush this may seem to be an unrealistic assumption, but in court all physical evidence is accompanied by some testimony authenticating or otherwise speaking to its relevance and veracity. For example, merely exhibiting a firearm during the course of a murder trial bears little relevance to the case at hand \textit{unless} someone testifies to salient facts concerning, for instances, its ownership, fingerprints found on the firearm, and the matching of the firearm to bullets recovered at the scene of the crime.

Second, I assume that the guilt of a defendant is \textit{materially} independent from the act of any witness testifying. This is of course not to say that testimony lends no inductive weight to the case at hand, but rather that such a testimonial act never \textit{necessitates} guilt or innocence.\footnote{This might not hold if, for instance, the defendant is charged with murder of person $A$ and then $A$ testifies during the trial. Then the act of $A$ testifying contradicts the guilt of the defendant.}

This is not to say that there are no logical inferences to be made regarding the \textit{probabilities} of guilt and the veracity of the testimonies, only that there is no direct deductive relation between them. 

\subsection{Fundamentals of Bayesian Inference}

To set the stage, we review the basics of the Bayesian account of rationality. Let $S$ be a collection of sentences containing ``true'' $\top$, ``false'' $\bot$, and closed under Boolean operations; namely, conjunction $\wedge$, disjunction $\vee$, and negation $\neg$. For the purposes of this paper we will assume that $S$ is finite.

We assume that the semantics of the sentences in $S$ is understood extensionally; in other words, the sentence $S$ is identified with the collection 
of worlds $w$ in some ambient universe of possible worlds $W$ such that $s$ holds true in $w$. Under this semantics we may think of a subset $T\subset S$ of sentences as its corresponding \textit{event} \[E_T = \{w\in W\,|\, T \text{ is true in } w\}.\]

The Bayesian model of rationality supposes that each agent $A$ is equipped at the outset with a \textit{prior} $\PP$, which is a certain kind of mathematical object that encodes the degree of belief, or \textit{credence},  they afford each sentence $s\in S$. The structure of this prior $\PP$ is that of a probability measure on $S$:

\begin{defn}
A probability measure on a set $S$ of sentences is a function
\[\PP: \mathcal{P}(S) \to [0,1] \]
such that
\begin{enumerate}
    \item $\PP(E_{\bot}) = 0$,
    \item $\PP(E_{\top}) = 1$, and 
    \item If $E_T \cap E_{T'} = \varnothing$, then 
    \[\PP(E_T)+\PP(E_{T'}) =  \PP(E_T \cup E_{T'}). \qedhere\]
\end{enumerate}
\end{defn}

The Bayesian model of rationality requires that as evidence accumulates, the agent $A$ updates her degrees of belief on the basis of taking conditional probabilities:

\begin{defn}
Let $E$ be an event and $\PP$ a prior. The \textit{posterior distribution} of $\PP$ given $E$ is the conditional probability measure defined on $\mathcal{P}(S)$ given by
\[\PP_E(T) = \frac{\PP(E\cap T)}{\PP(E)}. \qedhere \]

\end{defn}

For further details regarding the epistemic interpretation of conditional probabilities and the general theory of Bayesian Rationality, the reader is directed to the excellent survey by Earman \autocite{earman1992bayes}.

\subsection{\textit{State v. Skipper} and the Court's Error}

Recall from the introduction the argument presented by the Supreme Court of Connecticut against the use of Bayesian reasoning in the setting of a criminal trial:

\begin{quote}
Because Bayes' Theorem requires the assumption of a prior probability of paternity, i.e., guilt, its use is inconsistent with the presumption of innocence in a criminal case such as this.... If we assume that the presumption of innocence standard would require the prior probability of guilt to be zero, the probability of paternity in criminal cases would always be zero.... In other words, Bayes' Theorem can only work if the presumption of innocence disappears from  considerations. \autocites(at 623){1994skipper} 
\end{quote}

While I argue that this argument is incorrect, the precise \textit{way} in which it is incorrect motivates the definition of a juror's disposition. The court's argument seems to be:
\begin{enumerate}
\item \label{Premise:poi} (Presumption of Innocence) The defendant is to be presumed innocent until proven guilty.
\item \label{Premise:bayes} (Principle of Bayesian Inference) A Bayesian juror must update their beliefs according to Bayes' rule when presented with evidence during the course of the trial.
\item \label{Premise:exon} If the jury finds, after hearing a set of testimony $T$, the conditional probability of guilt given the testimony $\PP\left(E_G \,\middle\vert E_T\right) = 0$  then the jury must acquit.
\item \label{Premise:bayestriv} The presumption of innocence implies that any prior adopted by the jury must satisfy
\[ \PP(E_G) = 0. \]
\item (Conditioning) Conditioning by \textit{any} set of testimony $E_T$ will yield a posterior probability of $0$:
\[ \PP\left(E_G \,\middle\vert\ E_T\right) = 0\] 
\item (Conclusion) For any criminal case, the jury must acquit.
\end{enumerate}

That Premise 4 is a misunderstanding of Bayesian epistemology is discussed in \autocite{allen1994probability}, but the precise nature of this error is very illustrative of the difference between \textit{credence} and \textit{decision} that our framework of juror dispositions and Bayesian threshold jurors distinguishes between. On the court's view, in order for a Bayesian agent to presume innocence would mean that the Bayesian could not entertain the mere \textit{possibility} of guilt. This is duplicitous: criminal trials are predicated on countenancing the possibility of \textit{both} guilt and innocence at the outset. As we will see, Judge Richard Posner \textit{also} finds Premise 4 faulty, instead arguing that the Presumption of Innocence requires that $\PP(E_G) = \frac{1}{2}$. The remainder of this section will be spent analyzing his account of Bayesian threshold jurors.

\subsection{Posner's Even-Odds Proposal}

Among other things, Judge Richard Posner is in part known for his economic approach to the law. In particular, his analysis of factfinding centers on the use of Bayes' Theorem\autocite[1486]{posner1998evidence}. Posner does ``make clear at the outset that I do \textit{not} propose that juries or judges be instructed in the elements of Bayesian theory... The most influential model of rational decision making under conditions of ineradicable uncertainty... it can be of great help, as we shall see, in evaluating the rationality of rules of evidence.''\autocite[3]{posner1998evidence}
Nevertheless, Posner models jurors as agents whose credences form a probability measure, which are updated in light of new evidence stemming from the testimony offered during the course of a trial. Moreover, in this model Posner interprets the burden of persuasion and the burden of proof beyond a reasonable doubt probabilistically:
 
 \begin{quote}
    In the typical civil trial... it is enough to justify a verdict for the plaintiff that the probability that his claim is meritorious exceeds, however slightly, the probability that it is not... 
    
    Type I errors are more serious than Type II errors in criminal cases therefore are weighted more heavily in the former by the imposition of a heavy burden of persuasion on the prosecution... Judges when asked to express proof beyond a reasonable doubt as a probability of guilt generally pick a number between .75 and 0.95.\autocite[34-36]{posner1998evidence}
 \end{quote}
 
 Therefore, in the context of a criminal trial an ideal rational juror is modeled as:
 
 \begin{defn} A juror $j$ assessing the guilt $G$ of some defendant on the basis of testimony $T$  a collection of possible testimonial events is \textit{Bayesian} in case:
 is Bayesian just in case $j$
\begin{enumerate}[label=\roman*]
\item The juror has a prior probability measure $\PP:\mathcal{P}(W)\to [0,1],$ 
\item (Conditionalization) Assigns probability $\PP(E_G\,|E_T)$ to guilt when the juror has heard all and only the testimonies $T$, and 
\item ($\theta$-Verdict Rule) There is a fixed $\theta$ with $0.5 < \theta < 1$ such that the juror $j$ renders a conviction if and only if $\PP(E_G\,|\, E_T) \geq \theta$. \qedhere
\end{enumerate}  
\end{defn} 

It is important to note here that rendering a verdict of convict on the basis of exceeding a threshold $\theta$ requires some work to justify in the framework of classical decision theory since an explicit utility function is not presented. Somewhat recent results of Easwaran \autocite[828]{easwarantruthlove} provide such a utility function. The idea is to give a reward $R$ to an agent just in case that agent correctly believes that a proposition $P$ is true, and to give a penalty $-W$ if the agent incorrectly believes that $P$ is false. It is required that the agent believes at least one of $P$ or $\neg P$, but not both. In our current context we may interpret ``belief'' as ``votes to convict'' and ``disbelief'' as ``votes to acquit.''

\begin{defn}

A \textit{doxastic state} on a set $S$ of sentences closed under Boolean operations is a function
\[d:S\to \{0,1\} \]
such that $d(s) = 1$ implies $d(\neg s) = 0$.

Let $s$ be a proposition and $R,W > 0$ be real numbers. The $(R,W)$-weight of $s$ is defined as
\begin{equation*}
        \eta_{R,W} (s)= \begin{cases}
            R & \text{if $s$ is true}, \\
            -W & \text{if $s$ is false.}
        \end{cases}
\end{equation*}

The \textit{score} of a doxastic state is given by
\[ \sigma_{R,W}(f) = \sum\limits_{s\in S} d(s) \eta_{R,W}(s). \qedhere\]

\end{defn}

A doxastic state encodes the a binary belief function: if $d(s) = 1$ then the agent believes $s$, and if $d(s) = 0$ then the agent does \textit{not} believe $s$. The score of the doxastic state encodes the correctness of the agent's doxastic state.

Easwaran shows that 
\begin{thm} \autocite[828]{easwarantruthlove}
For a given probability function $\PP$, a doxastic state maximizes expected score iff it believes all propositions $s$ such that $\PP(s) > \frac{W}{R+W}$ and believes no propositions $s$ such that $\PP(s) < \frac{W}{R+W}$. Both believing and not believing are compatible with maximizing expected score if $\PP(s) = \frac{W}{R+W}$.
\end{thm} 

In our setting, a Bayesian juror votes to convict the defendant when the posterior probability of guilt exactly equal to the threshold. The above theorem says that either choice maximizes expected score. Thus, for a threshold $0<\theta<1$, a Posnerian juror maximizes expected $(1-\theta,\theta)$-score, and so Posnerian jurors are representable as a Bayesian agent with respect to \textit{some} utility function.

This choice of utility function, however, avails itself to criticism. One way to interpret the score function above in the judicial context is to identify with $W$ the average net social cost of wrongful conviction and $R$ the average net social benefit of a correct conviction. While simple in its expression, making decisions according to such a score function runs into some difficult challenges.

For instance, there is little reason to think that the overall values of $W$ and $R$ would be the same across different crimes. The variation amongst crimes for the values of $W$ and $R$ would therefore adjust the value of the probability threshold $\theta$. Therefore, if a juror is to render her verdict on the basis \textit{only} of the posterior probability of guilt, expected \textit{score} might be optimized but expected net social benefit will not.

One might object to this picture by saying that while it is true that the precise values of $W$ and $R$ vary from crime to crime, it is the judge that sentences the defendant and judges have a great deal of discretion in determining the sentence.
However, for many offenses mandatory minimum sentencing renders it impossible for the judge to appropriately calibrate the punishment of the defendant once convicted.

Beyond the Bayesian model of a juror's reasoning described above, Posner also proposes the following probabilistic interpretation of the Presumption of Innocence:

\begin{quote} 
Ideally we want the trier of fact to work from prior odds of 1 to 1 that the plaintiff or prosecutor  has a meritorious case... Although bias is clearest when the judge or jury not only has a prior belief about the proper outcome of the case but also holds the belief unshakably--that is, refuses to update it on the basis of evidence--it is not a complete response to a charge of bias that the judge or juror has an ``open mind'' in the sense of being willing to adjust his probability estimate in the light of the evidence presented at the trial. Any rational person will do that... His prior odds, if he is a Bayesian, will still have an influence on his posterior odds and hence... on his decision. \autocite[1514]{posner1998evidence}  
\end{quote} 

Posner's solution does guarantee a form of the Presumption of Innocence: provided that the threshold $\theta$ is chosen to be above $50\%$, no Bayesian juror conforming to the constraints that Posner outlines would convict absent any evidence. However, that is all it guarantees. Since testimony is presumed logically independent from the material facts at hand, it is perfectly consistent to ensure that no matter what testimony is afforded the juror will convict as soon as testimony of any kind is given. More formally:

\begin{propo}\label{prop:posner}
Let $0 < \theta < 1$ and that $T$ satisfies (TCA). Then there exists a prior probability $\PP$ such that 
\[\PP\left(E_G\right) = \frac{1}{2} \]
but for any nonempty collection $T$ of testimony
\[ \PP\left(E_G \,\middle\vert\, E_T\right) \geq \theta. \qedhere\]
\end{propo}
\begin{proof}

By the  it suffices to show that we can ensure that
\[ \PP\left(E_G \,\middle\vert\ E_T\right) \geq \theta,\]
but the conditional extension lemma (\autoref{lem:cond-ext}) ensures this.
\end{proof}

In other words, Posner's proposal--constraining only the \textit{priors} of the jurors--is only sufficient to guarantee that juror acting in accordance with Posner's rule will not convict at the outset of the trial, and moreover is compatible with \textit{guaranteed} conviction as soon as the first testimony is offered. By this result, constraining the prior probability of guilt to yield 1 to 1 odds only ensures that a juror's disposition cannot be to convict at the outset of the trial.

\section{Rationalizing Juror's Dispositions}

Having defined a formal model of a juror's reasoning in the preceding section, we are now in a position to evaluate whether Posner's view places any constraints on a threshold Bayesian juror's disposition beyond the Presumption of Innocence and the Willingness to Convict. 

Let $\theta \in (0,1)$. We call a juror disposition $f:\mathcal{T}_S \to \{C,A\}$ \textit{$\theta$-rationalizable} provided that there exists a prior $\PP_f$ on $\mathcal{P}(W)$\footnote{Where $W$ is constructed as in Section 2} such that
\[ f(T) = C \iff \PP_f(E_G|E_T) \geq \theta.\]
In other words, a disposition $f$ is $\theta$-rationalizable just in case the verdicts reached by $f$ on all transcripts $T$ can be realized as an instance of a threshold juror determining that the probability of the defendant's guilt meets or exceeds $\theta$ at a trial specified by transcript $T$.

The aim of this section is to prove a representation theorem that states that \textit{all} juror dispositions satisfying (PoI) and (WtC) are the dispositions of \textit{some} threshold Bayesian juror.

\begin{thm}\label{thm:rep_thm_2}
Suppose that $S$ is a finite collection of testimonies. Let
\[f: \mathcal{T}_S \to \{C,A\} \]
be a juror's disposition such that
\begin{enumerate} 
\item $f(\varnothing) = A$ (PoI)
\item there exists a transcript $T$ such that $f(T) = C$. (WtC)
\end{enumerate} 
Suppose that $\frac{1}{2} < \theta < 1$. Then there exists a prior $\PP_f$ on $W_S$ $\theta$-rationalizing $f$ such that
\[ \PP_f(E_G) = \frac{1}{2}.\]

Moreover, $\PP_f$ can be taken to be open-door in the sense that $\PP_f\left(E_G \,\middle\vert\ E_T\right)\notin \{0, 1\}$ for any transcript $T$.
\end{thm} 

This representation theorem shows that constraints set forth by Posner to analyze the efficiency of the trial system place no meaningful constraint whatsoever on the dispositions of the finders of fact in question beyond their nontriviality.

\begin{proof} Suppose that $f:\mathcal{T}_S \to \{C,A\}$ is a disposition satisfying the hypotheses of the theorem statement and that $\frac{1}{2} < \theta < 1$. The support of $C$ (resp. $A$) is the set $f^{-1}(C) = \{T\in \mathcal{T}_C\,|\, f(T) = C\}$ (resp. $f^{-1}(A)$). By definition, $f^{-1}(C)$ and $f^{-1}(A)$ partition $\mathcal{T}_C$, and by the assumption of the theorem they are nonempty.

We define $\PP_f$ as a weighted combination of two measures defined in terms of the supports of $C$ and $A$. Let $n_C = |f^{-1}(C)|$ and $n_A = |f^{-1}(A)|$.

Set \[\PP_{f,C}(\{(T,x)\}) = \begin{cases}
        \theta n_C^{-1} & \text{if }T\in f^{-1}(C) \text{ and } x=G, \\
        (1-\theta)n_C^{-1} & \text{if }T\in f^{-1}(C) \text{ and } x=I, \\
            0 & \text{if }T\notin f^{-1}(C).
        \end{cases} \]
and \[\PP_{f,A}(\{(T,x)\}) = \begin{cases}
        (1-\theta)n_A^{-1}  & \text{if }T\in f^{-1}(A) \text{ and } x=G, \\
        \theta n_A^{-1} & \text{if }T\in f^{-1}(A) \text{ and } x=I, \\
            0 & \text{if }T\notin f^{-1}(A).
        \end{cases}.\]

Both $\PP_{f,C}$ and $\PP_{f,A}$ induce probability measures on $W_S$. Since 
$\PP_{f,C}(E_G) = \frac{\theta n_C}{n_C} = \theta > \frac{1}{2}$ and $\PP_{f,A}(E_G) = \frac{(1-\theta) n_A}{n_A} = 1-\theta < \frac{1}{2}$ there exists some $\alpha \in (0,1)$ such that the measure
\[ \PP_f = \alpha \PP_{f,C} + (1-\alpha) \PP_{f,A} \] 
has $\PP_f(E_G) = \frac{1}{2}$.

Moreover,\[ \PP_f(E_G|E_T) = \frac{\PP_f(E_G \cap E_T)}{\PP_f(E_T)} = \frac{\PP_f(T,G)}{\PP_f((T,G),(T,I))} = 
\begin{cases}
        \theta & \text{if }T\in f^{-1}(C) \\
        (1- \theta) & \text{if }T\in f^{-1}(A)\\
        \end{cases}\]
ensuring that $\PP_f$ $\theta$-rationalizes $f$ and is open-door.\end{proof}

\section{Further Constraints on Juror's Dispositions?} 
	
Having seen that Posner's proposal puts no constraints on a juror's disposition beyond the Presumption of Innocence and the Willingness to Convict, one may attempt to rescue this account by placing further constraints on the definition of a Bayesian threshold juror to pare down the class realizable juror dispositions. The final section of this paper is intended to survey obstacles that possible additional constraints on Bayesian accounts of rational jurors face. Broadly speaking, the additional constraints fall into three separate categories:
\begin{enumerate}
    \item The \textit{uniformity} of the prior with respect to a well-chosen sample space,
    \item The \textit{objectivity} of the prior, and
    \item The rate of convergence to a verdict of ``Guilty'' of a prior.
\end{enumerate}
I argue that, ultimately, none of these constraints suffice to save the Bayesian account of rational jurors. 

\subsection{Constraint One: Mandating Uniform Priors}
 
The first potential save we will consider is to restrict the class of Bayesian threshold jurors to include only uniform priors with respect to an appropriate sample space. At first glance, a uniform prior of guilt on some (appropriately large) collection $X$ of persons is appealing. For sake of simplicity, suppose that only one person $x$ is guilty of the crime in question; that is, $E_G = \{x\}$. Then
\[ \PP(E_G) = \frac{1}{|X|}.\]
So long as $X$ has at least two members and the threshold $\theta > \frac{1}{2}$, a juror with this prior will not convict at the outset of the trial. While this is a perfectly well-defined prior, one runs into trouble with how a juror is to update her credences on the basis of witness testimony. After all, when the axiom TCA holds, the occurrence of testimony $T$ is consistent with the guilt or innocence of \textit{each} $y\in X$. This implies that the set of possible worlds in which testimony $T$ occurs is not expressible as a subset of $X$. 

Faced with this obstruction, an advocate for assigning a uniform prior probability of $\frac{1}{|X|}$ to guilt must either supply us with a sufficiently rich sample space that the ideal juror ought to update her credences according to or argue that constraining the numerical probabilities to be proportional to the size of the set of potential perpetrators is sufficient to constrain the ideal juror. The latter option is untenable, as the proof of \autoref{prop:posner} shows: merely constraining the numerical prior probability of guilt to be low poses no constraint on how rapidly the juror may converge to an assessment of guilt. 

As an illustration, suppose that the sample space $X$ consists of the residents of Manhattan. A witness testifies that the perpetrator of the crime has brown eyes. Call this testimonial event $T$. If the witness is taken to be certainly correct, the event $E_T$ is
\[E_T = \{x\in X\,|\, \text{$x$ has brown eyes} \} \subseteq X,\]
and the updated probability of the guilt of the defendant $d$ is:
\[\PP(E_G \,|\, E_T) = \begin{cases} 
      0 & d \notin E_T \\
      \frac{1}{|E_T|} & d\in E_T 
   \end{cases} \]
However, if it is acknowledged that the witness might be mistaken, the testimony does not conclusively rule out any member of $X$; in other words, 
\[E_T = X.\]
Therefore the sample space $X$ has insufficient expressive power to facilitate nontrivial probabilistic reasoning.

\subsection{Constraint Two: Requiring Objectivity of a Juror's Prior and the Principal Principle}

A second constraint one may impose on a Bayesian threshold juror is that the juror's prior be in some sense \textit{objective}. A convenient way to formalize this is bay way of Lewis' Principal Principle, which can be expressed as follows:

\begin{quote} \textbf{Principal Principle:} Assume we have a number $x$, proposition $A$, time $t$, rational agent whose evidence is entirely about times up to and including $t$, and a proposition $E$ that (a) is about times up to and including $t$ and (b) entails that the chance of $A$ at $t$ is $x$. In any such case, the agent's credence in $A$ given $E$ is $x$.  \autocite{sep-david-lewis}
\end{quote} 

The difficulty with this approach is that, in the context of a trial, a Bayesian Threshold juror will not witness an event $E$ of the sort referenced in the statement of the Principal Principle.

An excellent example of this exact issue playing out in case law can be found in \textit{State v. Spann} \autocite{1993spann}. During the course of the trial an expert witness testified, on the basis of Bayesian analysis, that from a prior probability of 50\% of paternity the defendant's blood test rendered a 96.55\% probability of paternity upon posterior updating. The court recounts the cross-examination of an expert witness by defense's counsel:
\begin{quote} 
On cross-examination defense counsel brought out the fact that the probability of paternity percentage was based on that fifty-fifty assumption. The expert described it as a ``neutral'' assumption... [h]er characterization of the evidence was that its ``purely objective'' nature was ``one of the beauties of the test''; that it ``makes no assumption other than everything is equal''; and that ``the jury simply has objective information'' ... Counsel noted that even if it were conclusively proven that defendant had been out of the country at the time when conception could have occurred, this expert still would have concluded that the probability the defendant was the father was 96.55\%. Counsel's observation was correct; the expert's opinion had no relation whatsoever to the the facts of the case. \autocite[590]{1993spann}
\end{quote}

Defense's observation was astute. There was no \textit{mathematical} error in the expert testimony, being a straightforward application of Bayes' Theorem. The expert's reliance on the ``fifty-fifty'' assumption underlies an even deeper issue: the expert testimony imposes undue constraints on the \textit{structure} of the juror's priors \textit{beyond} the simple constraints of ``guilty'' vs. ``not guilty.'' Proper application of Bayesian updating requires knowledge of the \textit{full} structure of the juror's prior, incorporating not only their prior assessment of guilt but also the other pieces of testimony they had heard, their background assumptions regarding the veracity of expert testimony, their understanding of general causal laws, and the like. 

We reconstruct the probabilistic analysis we see in \textit{State v. Spann}. Suppressed in the testimony is an underlying \textit{sample space}. The basic data of the sample space is a tuple:
\[(b_X, b_Y, p(X,Y))\]
where $b_X, b_Y$ represent the blood types of a pre-selected pair of people, $X$ and $Y$, each a member of the following: $\{A+, A-, AB+, AB-, B+, B-, O+, O- \}$ and $p$ represents the paternal relationship between the two; either ``True'' if the first person is a parent of the second, and ``False'' otherwise. This yields a sample space of size
\[ 8 \times 8 \times 2 = 128.\]
The event ``$X$ is a parent of $Y$'' has size 64 in the sample space---one half the size of the total---and the expert advises us to adopt the prior that 
\[\PP(\text{``$X$ is a parent of $Y$''}) = \frac{1}{2}.\]
This probability model cannot update on an event of the form ``$X$ was in a different country from $Y$'s mother for the 5 years before and after $Y$'s birth'' as it is not coextensive with any subset sample space: the sample space is far too coarse. Thus the expert testimony carries with it a suppressed underlying model of the possible states of the world, which with good reason are insufficient for the purpose of updating in a Bayesian manner.

Presumably, an expert advising the jury on the probability $p$ of the defendant's guilt on the basis of evidence such as DNA matching intends for their testimony would move the juror's credence of guilt to be $p$, absent any other testimony. A juror's updating of the prior in this way would, however, generally not be an instance of the Principal Principle. After all, the agent conditionalizes on testimonial data of the form ``the expert testified that the probability of the defendant's guilt is $p$,''
which by the TCA axiom is logically independent from the guilt of the defendant. 

The obstruction to applying the Principal Principle in this case is the missing premise that \textit{if} the expert testifies that the probability of the defendant's guilt is $p$ \textit{then} the probability of the defendant's guilt is $p$, a premise which would refute the TCA.

\subsection{Constraint Three: Objective Relevance Standards and Objective Likelihood Ratios}

Another potential constraint one might consider to further constrain a Bayesian threshold juror is to require that the juror's conditional updating is compatible with some notion of an objective likelihood ratio. 

Posner advances such an argument, but as we will see there is an awkward tension in Posner's analysis between objective and subjective probabilities. One might get the impression from Posner's description that there are objective probabilities governing the computation of likelihood ratios, as in the following excerpt:

\begin{quote}
    Suppose that the new piece of evidence is testimony by bystander $Z$ that he saw $X$ shoot $Y$. Suppose further that the prior odds $\Omega(H)$ are 1 to 2 that $X$ shot $Y$, while the probability that $Z$ would testify that he saw $X$ shoot $Y$ if $X$ did shoot $Y$ is $.8$ and the probability that he would testify that he saw $X$ shoot $Y$ if $X$ did not shoot $Y$ is $.1$, so that the likelihood ratio is 8. The posterior odds that $X$ shot $Y$ will therefore be 4 to 1...
    
    [A]ltering posterior odds may not have much or even any social value even if the likelihood ratio of the new evidence is high, as in our shooting example, where it was 8. The value of the evidence will depend on the prior odds and on the decision rule. Suppose that the prior odds (as a consequence of the previously presented evidence) that $X$ shot $Y$ are not 1 to 2 but 1 to 10 and that for $X$ to be held liable for the shooting the trier of fact must consider the odds that he did it to be at least 1.01 to 1 (the preponderance standard). Then the new evidence, since it would lift the posterior odds above the threshold (multiplying the prior odds by a likelihood ratio of 8 yields posterior odds of only 1 to 1.25), would have no value.
    \autocite[1486-7]{posner1998evidence}
\end{quote}

It is important to note that in general one cannot posit an agent-independent likelihood ratio of a piece of evidence $E$: if an agent's prior probability of hypothesis $H$ is $\PP(H) = \theta$, then any agent with prior probability distribution $\PP$ is assured that the likelihood ratio is at most $\frac{1}{\theta}$. In other words, the likelihood ratio afforded to evidence is \textit{inseparable} from the structure of the agent's prior probability \textit{measure} $\PP$, not just its numerical values. Therefore an attempt to save Posner's account on the basis of something like objective likelihood ratios is doomed to fail: any prescribed value $L(E,H)$ will result in inconsistent assessments of probability for many agents.

The strange hybrid of subjective and objective probabilities appears again in Posner's account of the relevance standard of the Federal Rules of Evidence (FRE 401), where Posner interprets it within an economic framework: ``In Bayesian terms, evidence is relevant if its likelihood ratio is different from one and irrelevant if it is one.'' \autocite[1522]{posner1998evidence} This Bayesian gloss fails to emphasize that the assessment of whether or not a piece of evidence is relevant in the sense that its likelihood ratio is different from 1 depends on the structure of the factfinder's prior probability measure. 

To see this, note that having a likelihood ratio $L_{\PP}(E,H)$ equal to 1 is equivalent to evidence $E$ being probabilistically independent from $H$. The measure extension lemma---\autoref{lem:cond-ext}---entails that if a piece of evidence $E$ is logically independent from all preceding evidence, then there exist probability distributions $\PP$ in which $E$ is probabilistically independent of the rest of the evidence, i.e. \textit{irrelevant}, and probability distributions in which the evidence is probabilistically dependent, i.e. \textit{relevant}. On Posner's account, FRE 401 is at best underspecified, and at worst fangless: all testimonial evidence is both \textit{potentially} relevant and \textit{potentially} irrelevant. Worse yet, even if one constrains the class of relevant testimony, the effect on the posterior probability is unconstrained.

\subsection{Constraint Four: Restricting Convergence Rates}

The final constraint we consider in this section is that one might attempt to save the Bayesian account by requiring that a juror not be too quick in reaching a verdict by limiting the degree to which any given piece of testimony can affect a juror's beliefs. For instance, one might demand that the \textit{ratio} between prior and posterior belief in guilt is bounded by some fixed amount for all testimonies $T$, e.g. by requiring that
\[ \frac{2}{3} \leq \frac{\PP(E_G|E_{t_1}\cap \dots\cap E_{t_m+1} )}{\PP(E_G|E_{t_1}\cap \dots\cap E_{t_m})} \leq \frac{3}{2}.\]
For a juror who convicts only when $\PP(E_G|E_{t_1}\cap \dots\cap E_{t_m} \cap E_{t_m+1}) > \theta > \frac{3}{4}$ and for whom $\PP(E_G) = \frac{1}{2}$, such a rule would ensure that, at least conceivably, \textit{one} piece of testimony at any time would be insufficient to render a verdict of ``Guilty'' since $\PP(E_G|E_{t_1}) < \frac{3}{4} < 1$.

That said, for any rule of the form
\[ \frac{\PP(E_G|E_{t_1}\cap \dots\cap E_{t_m+1} )}{\PP(E_G|E_{t_1}\cap \dots\cap E_{t_m})} \leq 1+ \gamma\]
where $\gamma > 0$, verdict threshold $\theta > 0$, and prior probability of guilt $\PP(E_G) = \frac{1}{2}$ there exists jurors who will convict just so long as there are 
\[m > \frac{log(2\theta)}{\log(1+\gamma)}\]
pieces of testimony, a disposition Posner would surely want no juror to have.

Similarly, it is easily conceivable that we would \textit{want} a juror to be able to convict having witnessed a single piece of testimony. For example, consider a defendant who---defying their pretrial pleading of ``Not Guilty''---has a sudden change of heart on the stand and confesses to the crime at hand, a juror can hardly be faulted in deciding to convict on that basis alone.
 
\subsection{Outlook on Further Constraints}

The candidate constraints considered in this section range from constraints regarding the uniformity, objectivity, and the rate of convergence of a Bayesian threshold juror's prior to conviction. Each of these candidates faced severe challenges, and either were mathematically ill-defined, were not operationalizable in the trial setting, or constrained the class of rational dispositions too much. I do not claim that this list of candidate constraints is exhaustive; however, the prospects of a formal solution to the problems with the threshold Bayesian account laid out by the representation theorem appear bleak.

\section{Conclusion} 

In the Introduction to this chapter, we saw a deep divide between two factions. On one side we have the anti-Bayesian current, with Tribe and the judges of the Supreme Court of Connecticut the vanguard members advocating for the \textit{inadmissibility} of Bayesian and other probabilistic forms of reasoning from the criminal trial system on the basis of a perceived conflict with the Presumption of Innocence. The opposing, Bayes-rationalist side of this dispute, exemplified by Judge Posner, claim that to the contrary that rationality \textit{requires} these forms of reasoning, lest the criminal justice system fall victim to a strain of irrationality.

The analysis of this paper suggests a mundane resolution to this dispute: there is neither harm in \textit{nor} necessity to demand a juror be Bayes rational; so long as a juror's disposition satisfies the Presumption of Innocence and the Willingness to Convict, that juror's disposition is indistinguishable from a Bayesian threshold juror's regardless of the underlying causal source of her dispositions. This result strikes at the heart of both the strongly anti-Bayesian and pro-Bayesian accounts: if you demand that Bayesian inference be banned in all its forms, there is no way to discern this on the basis of an agent's dispositions. Likewise, for the pro-Bayesian account, the representation theorem demonstrates that \textit{nothing is gained by demanding that an agent be Bayes rational}. Thus, Posner's notion of a Bayesian juror is insufficiently specified to render this debate a substantive one.

\section{Appendix: Lemmata from Probability Theory.}

A very general extension theorem \autocite[70]{rao1983charges} goes as follows:

\begin{thm} 
Let $\CC$ be a Boolean algebra of subsets of a set $\Omega$ and let $\mu: \CC \to [0,1]$ be a finitely additive, positive bounded measure. Suppose that $A \in \PP(\Omega) \setminus \mathcal{C}$. Write
\[\mu_{l}(A) = \sup\{\mu(B) \,|\, B\in \mathcal{C} \wedge B\subset A \} \]
and
\[ \mu_{u}(A) = \inf\{\mu(B) \,|\, B\in \mathcal{C} \wedge B\supset A \}.\]
Then for any $d\in [\mu_{l}(A), \mu_{u}(A)]$\footnote{Note that $\mu_{l}(A) \leq \mu_{u}(A)$ so there is always at least one extension to the measure $\mu$.}
there exists a finitely additive, positive bounded measure
\[\tilde{\mu}: \mathcal{C}\left\langle A \right\rangle \to [0,1] \]
such that 
\[ \tilde{\mu}(A) = d \qedhere \]
\end{thm}

This lemma gives necessary and sufficient conditions to extend \textit{finitely additive} measures to larger Boolean algebras. 

More generally, we have a great deal of control over extending measures to ensure certain conditional probability inequalities hold:

\begin{defn}
Let $\CC$ be a Boolean algebra on $X$. We say that $B$ is logically independent from $\CC$ provided that for all $A\in \CC$, \[A \neq X,\varnothing \rightarrow (A\cap B \neq \varnothing \wedge A^c \cap B \neq \varnothing) \qedhere \]

\end{defn}

In other words, $B$ intersects every nontrivial Boolean combination of elements of $\CC$ nontrivially. When construing the set $X$ as a set of possible worlds, this is the same as proposition $B$ being logically independent of any set of propositions in $\CC$.

\begin{lemma} \label{lem:cond-ext}
Suppose that $A\in\CC$, $\PP(A) \notin \{0,1\}$, $\PP$ a probability charge on $\CC$ and $B$ is logically independent from $\CC$. Then for all $\theta \in [0,1]$ there is an extension $\widetilde{\PP}$ of $\PP$ to $\CC\left\langle B\right\rangle$ such that 
\[ \widetilde{\PP}(A|B) = \theta. \qedhere\]
\end{lemma}

\begin{proof}
We apply the extension theorem twice: write
\[\PP(B) = \PP(B\cap A) + \PP(B \cap A^c)\]
and for any $C\in \CC$ let $\PP_C(D) = \PP(D\cap C)$. This is an unnormalized measure on the algebra $\CC_{C} = \{C\cap D\,|\, D\in \CC\}$. When $\PP(C) \notin \{0,1\}$ the measures $\PP_C$ and $\PP_{C^c}$ are both positive bounded charges on $\CC_C$ and $\CC_{C^c}$. Now as $B$ is logically independent from $\CC$ the set $B\cap A$ (resp. $B\cap A^c$) is logically independent of $\CC_{A}$ (resp. $\CC_{A^c}$). By the charge extension theorem we may extend the charge $\PP_A$ (resp. $\PP_{A^c}$) to a charge $\widetilde{\PP_{A}}$ on $\CC_A\left\langle B\cap A\right\rangle$ (resp. $\widetilde{\PP_{A^c}}$ on $\CC_A\left\langle B\cap A^c\right\rangle$) assigning to it any value $\rho_0 \in [0,\PP(A)]$ (resp. $\rho_1 \in [0,1-\PP(A)]$). Since the algebras $\CC_A\left\langle B\cap A\right\rangle$ and $\CC_A\left\langle B\cap A^c\right\rangle$ are disjoint, we may define a probability charge $\widetilde{\PP}$ on $\CC\left\langle B\right\rangle$ via the formula
\[\widetilde{\PP} = \widetilde{\PP_{A}} + \widetilde{\PP_{A^c}}.\]

Let $\theta \in [0,1]$. Then we wish to express $\theta$ as
\[\theta = \widetilde{\PP}(A|B) = \frac{\widetilde{\PP}(A\cap B)}{\widetilde{\PP}(A\cap B) + \widetilde{\PP}(A^c\cap B)} = \frac{\rho_0}{\rho_0 + \rho_1} \]
This is a continuous function in $(\rho_0,\rho_1)$ taking on the values of $0$ ($\rho_0$ = 0 and $\rho_1 = \PP(A^c)$) and $1$ ($\rho_0 = \PP(A)$ and $\rho_1 = 0$) and so by taking a suitable line inside $[0,\PP(A)] \times [0,\PP(A^c)]$ one can apply the Intermediate Value Theorem to conclude that there are $\rho_0,\rho_1$ rendering
\[ \PP(A|B) = \theta \]
true.
\end{proof}

This proposition tells us that when adjoining a logically independent event to our algebra $\CC$, $\PP$ may be consistently extended in such a way as to render the conditional probability $\widetilde{\PP}(A|B)$ \textit{any} value whatsoever. This indicates that merely placing constraints on the \textit{numerical} values of prior probability of an event $A$ places \textit{no} constraint on the posterior probability when updating by a logically independent event except in the degenerate cases where $\PP(A) \in \{0,1\}$.

\section{Appendix: Remarks on the Utility Functions Rationalized by Threshold Bayesian Jurors}

Easwaran's score function reflects a coarse grain of accuracy. In this section we generalize his result by considering the case of epistemic scores determined by the four possible outcomes of a juror's verdict: the conviction of a guilty defendant, the conviction of an innocent defendant, the acquittal of a guilty defendant, and the acquittal of an innocent defendant. Let 
\[ S = \{GC, NGC, GA, NGA\} \]
be the state space corresponding to these four possible outcomes. A collection $\{(s,\alpha_s)\}_{s\in S}$ with each $\alpha_s \in \RR$ determines a utility function:
\[ U(s) = \sum\limits_{t\in S} \chi_t(s) \alpha_s \]
where \[\chi_t(s) = \begin{cases}
        1 & \text{if }s=t, \\
            0 & \text{otherwise.}
        \end{cases} \]
Relative to a probability measure $\PP$ on $\{G,NG\}$, the expected utility of each verdict is given by:
\[E(U;C) = \PP(E_G) \alpha_{GC} + (1-\PP(E_G))\alpha_{NGC},\]
and
\[E(U;A) = \PP(E_G) \alpha_{GA} + (1-\PP(E_G))\alpha_{NGA}.\]
To maximize expected utility, an agent will convict just in case
\[ E(U;C) \geq  E(U;A).\]
Simple algebraic manipulation shows that this occurs precisely when
\[\PP(E_G) \geq \frac{\alpha_{NGA} - \alpha_{NGC}}{\alpha_{GC}-\alpha_{NGC} - \alpha_{GA} +\alpha_{NGA}}.\]
This computation entails that being a threshold juror is rationalizability according to more flexible epistemic utility functions.

\printbibliography

\end{document}